\documentclass[12pt]{amsart}
\usepackage{color}
\usepackage[all]{xy}
\usepackage{amssymb,amsmath}
\usepackage{amscd,amsmath,latexsym,amssymb}
\usepackage[mathcal]{euscript}

\date{August 19, 2020}

\calclayout
\newtheorem{dummy}{anything}[section]
\newtheorem{theorem}[dummy]{Theorem}
\newtheorem*{thma}{Theorem A}
\newtheorem*{thmb}{Theorem B}

\newtheorem{lemma}[dummy]{Lemma}
\newtheorem{proposition}[dummy]{Proposition}
\newtheorem{corollary}[dummy]{Corollary}

\theoremstyle{definition}
\newtheorem{definition}[dummy]{Definition}
  \newtheorem{example}[dummy]{Example}
  
  \newtheorem{remark}[dummy]{Remark}
   \newtheorem*{nonum}{Theorem}
   
    \newtheorem*{question}{Question}
    \newtheorem*{Conjecture}{Conjecture}
  \newtheorem*{acknowledgement}{Acknowledgement}

\newcommand
{\eqncount}{\setcounter{equation}{\value{dummy}}%
\addtocounter{dummy}{1}}

\newcommand{\cF}{\mathcal F}

\newcommand{\bF}{\mathbf F}

\newcommand{\bk}{\mathbf k}

\newcommand{\bZ}{\mathbb Z}

\newcommand{\CP}{\mathbb C \mathbb P}

\newcommand{\bbC}{\mathbb C}
\newcommand{\bbF}{\mathbb F}

\newcommand{\bbR}{\mathbb R}

\newcommand{\bbZ}{\mathbb Z}

\DeclareMathOperator{\Hom}{Hom} 

\DeclareMathOperator{\rk}{rank}
\DeclareMathOperator{\image}{Im}
\DeclareMathOperator{\Ext}{Ext}

\DeclareMathOperator{\Res}{Res}

\DeclareMathOperator{\coker}{coker}

  \DeclareMathOperator{\Fix}{Fix}
  \DeclareMathOperator{\Tors}{Tors}

\newcommand{\cy}[1]{\bZ_{#1}}

\newcommand{\la}{\langle}
\newcommand{\ra}{\rangle}

\newcommand{\disjointunion}{\bigsqcup}

\newcommand{\vv}{\, | \,}

\newcommand{\Fp}{\bbF_p}

\newcommand{\ZG}{\bZ G}

\newdir{ >}{{}*!/-5pt/\dir{>}}
\newdir{ (}{{}*!/-5pt/^\dir{(}}

\newcommand{\RP}{\bbR \mathbb P}
\DeclareMathOperator{\Ess}{Ess}

\begin{document}

\title{Rank conditions for finite group actions on $4$-manifolds}
\author{Ian Hambleton}
\author{Semra Pamuk}

\address{Department of Mathematics, McMaster University
L8S 4K1, Hamilton, Ontario, Canada}

\email{hambleton@mcmaster.ca }

\address{Department of Mathematics, Middle East Technical University, 06800, Ankara, Turkey }

\email{pasemra@metu.edu.tr }


\thanks{This project was partially supported by NSERC Discovery Grant A4000. The first author wishes to thank the Max Planck Institut f\"ur Mathematik for its hospitality and support in November 2019. The second author was supported on sabbatical by METU, and thanks McMaster University for its hospitality during the academic year 2018-19.}

\begin{abstract}
Let $M$ be a closed, connected, orientable topological $4$-manifold, and $G$ be a finite group acting topologically and locally linearly on $M$. In this paper we investigate the spectral sequence for the
Borel cohomology $H^*_G(M)$, and establish new bounds on the rank of $G$ for homologically trivial actions with discrete singular set.
\end{abstract}

\maketitle

\section{Introduction}
\label{sec:one} 

In this paper we provide some new information about the existence of finite group actions on closed, connected, orientable $4$-manifolds. In this dimension, the comparison between smooth and topological group actions is particularly interesting. Our focus  will be on \emph{locally linear topological actions} as background for future work on smooth actions.

For free actions on simply-connected $4$-manifolds, or equivalently for closed topological $4$-manifolds with finite fundamental group, there are a number of classification results in the literature (for example, see \cite{hk2,hk6,hk5,hkt1}). One challenging open problem is to compute the Hausmann-Weinberger invariant \cite{Hausmann:1985}, namely to determine the minimal Euler characteristic of a $4$-manifold with a given fundamental group. The answer is only known at present in special cases (see \cite{Kirk:2009}).

We will extend the scope of previous work by including  \emph{non-simply-connected} manifolds, and concentrate on non-free actions. We often assume that the actions are \emph{homologically trivial}, meaning that the group of symmetries acts as the identity on the  homology groups of the manifold. 

 A useful measure of the complexity of a finite group $G$ is its \emph{$p$-rank}, defined as the maximum rank $r$ of an elementary abelian $p$-group $(\cy p)^r  \leq  G$. We let  $\rk_p(G)$ denote the $p$-rank of $G$ for each prime $p$, and  let  $\rk (G)$ denote the maximum over all primes  of the $p$-ranks.
 
\begin{question}
Given a closed orientable $4$-manifold $M$, what is the maximum value of $\rk(G)$ over all the finite groups $G$ which act effectively, locally linearly,  and homologically trivially on $M$ ?
\end{question}

We note that a $\cy p$-action,  for $p$ a prime, will be homologically trivial if  $M$ has torsion-free homology and $(p-1)$ is larger than each of the Betti numbers  of $M$. If $M$ has Euler characteristic $\chi(M) \neq 0$, then any homologically trivial action of a finite group must be non-free (by the Lefschetz fixed-point theorem). 

Beyond the rank restrictions, we would like to know which finite groups $G$ can act. For example, if  $M$ is the connected sum of two or more complex projective planes, then $G$ is abelian and  $\rk(G) \leq 2$ (see \cite{hlee3}). This was proved for smooth actions using techniques from gauge theory.  Then  McCooey \cite{McCooey:2002a}, building on earlier work by Edmonds \cite{Edmonds:1998},   used  methods from equivariant algebraic topology to prove a much stronger result:

\begin{nonum}[{\rm McCooey \cite[Theorem 16]{McCooey:2002a}}]
Let $G$ be a (possibly finite) compact Lie group, and suppose $M$ is a closed $4$-manifold with $H_1(M;\bZ) = 0$ and $b_2(M) \geq 2$, equipped with an effective, locally linear, homologically trivial $G$-action.
\begin{enumerate}
\item If $b_2(M) =2$ and $\Fix(M)\neq \emptyset$, then $G$ is isomorphic to a subgroup of 
$S^1 \times S^1$.
\item If $b_2(M) \geq 3$ then $G$ is isomorphic to a subgroup of 
$S^1 \times S^1$, and the fixed set $\Fix(M)$ is necessarily
nonempty.
\end{enumerate}
\end{nonum}

What should we expect for actions on arbitrary \emph{non-simply-connected} $4$-manifolds~? Here is possible uniform answer to the rank question (compare \cite[Conjecture 9.1]{Edmonds:1998}).
\begin{Conjecture}  If a finite group $G$ acts effectively, locally linearly,  and homologically trivially on  a closed  orientable $4$-manifold $M$ with Euler characteristic $\chi(M) \neq 0$, then $\rk_p(G) \leq 2$ for $p$ odd.
\end{Conjecture}

The condition $\chi(M) \neq 0$ rules out actions on $M = T^4$ (for example), but the group $G = (\cy 2)^4$ acts linearly on $S^4$, so additional conditions must be found for $p=2$.

\begin{remark}
Mann and Su \cite[Theorem 2.2]{Mann:1963}) showed that $\rk_p(G)\leq 2$, for $p$ an odd prime, provided that the fixed set $\Fix(M) \neq \emptyset$,  without assuming that the action was locally linear or homologically trivial. However, the existence of a global fixed point is a strong assumption:
 in the locally linear case the result follows easily from a result of P.~A.~Smith  \cite[\S4]{Smith:1960} applied to the boundary $S^3$ of a $G$-invariant $4$-ball at a point $x\in \Fix(M)$. 
\end{remark}

The main tool from equivariant algebraic topology used for the study of non-free group actions is the Borel spectral sequence. Let $BG$ denote the classifying space for principal $G$-bundles, and $EG$  the universal free, contractible $G$-space. Then the Borel cohomology $H_G^*(M) := H^*(M \times_G EG)$ is ``computable"  in principle from the Serre spectral sequence of the fibration $M \to M \times_G EG \to BG$.  
We will use integral coefficients or $\Fp$-coefficients for $H^*(M)$, but note that in general this is a local coefficient system for the group cohomology of $G$. For homologically trivial actions, we have ordinary coefficients.


 \begin{thma}
 Suppose that $G=\cy p$ acts locally linearly on a closed, connected, oriented $4$-manifold $M$, preserving the orientation, with fixed point set $F = \Fix(M) \neq \emptyset$.    \begin{enumerate}
 \item  If the map $H_1(F;\bZ)\twoheadrightarrow H_1(M;\bZ)$ is surjective, then the Borel spectral sequence for $H^*_G(M)$ collapses with integral and  $\Fp$ coefficients.
 \item If $\ker(H^1(M;\bZ) \to H^1(F;\bZ))$ is non-trivial, but has trivial $G$-action, then the Borel spectral sequence with integral coefficients does not collapse.
 \end{enumerate}
 \end{thma}

  Edmonds \cite[Prop.~2.3]{Edmonds:1989} showed that the Borel spectral sequence with integral or $\Fp$ coefficients collapses for orientation preserving  $\cy p$  actions with  $F\neq \emptyset$ on closed simply-connected $4$-manifolds.  We generalize this result  to non-simply-connected 4-manifolds. 

The $2$-dimensional components of $\Fix(M)$ are always orientable if $p >2$, and for $p=2$ this would follow, for example, by assuming that $G$ preserves a $Spin^c$ structure in a suitable sense (see \cite[Proposition 3.2]{Edmonds:1989} and Ono \cite[Section 4]{Ono:1993}). However, the complex conjugation involution on $\CP^2$ with fixed set $\RP^2$ shows that orientability of the fixed set is not necessary in general for the collapse of the Borel spectral sequence  with integral coefficients (see \cite[Prop.~2.3]{Edmonds:1989}).

%

In his arXiv paper  \cite[Proposition 3.1]{McCooey:2013} McCooey proposed  a  ``collapse" result for homologically trivial actions under the assumption that $H_1(M)$ is torsion-free, but without our condition on $H_1(F)$ (see Example \ref{ex:two} for a counter-example).

\begin{remark}
 Note that if   $H_1(F)\twoheadrightarrow H_1(M)$ is surjective, then $H^1(M)\rightarrowtail H^1(F)$ is injective, but not conversely if $H_1(M)$ has torsion. 
\end{remark}

Recall that an action is called \emph{pseudofree} if the singular set $\Sigma: = \Sigma(M, G) \subset M$ consists of isolated points. For such actions, we can estimate the rank.

\begin{thmb} Let $M$ be a closed, orientable $4$-manifold with $\chi(M) \neq 0$. If a finite group $G$ acts locally linearly, pseudofreely and homologically trivially on $M$, then $\rk_p(G) \leq 1$
 for  $p\geq 5$, and $\rk_p (G) \leq 2$ for $p = 2, 3$.
 \end{thmb}

\begin{remark} Note that the actions of $G = (\cy 2)^4$ on $M = S^4$ are not pseudofree (see Breton \cite{Breton:2011}). In addition, $M = \CP^2$ admits a pseudofree action of $G = \cy 3 \times \cy 3$, (see Example \ref{ex:cp2}), 
 and $S^2 \times S^2$ admits pseudofree actions of $\cy 2 \times \cy 2$ (see \cite{McCooey:2007a}).
\end{remark}

Here is a short outline of the paper. Throughout the paper $M$ denotes a closed, connected, oriented topological $4$-manifold. 

 For orientation-preserving actions,  the assumptions in Theorem A imply that  the fixed point set must be two dimensional whenever $H_1(M) \neq 0$. In  Section \ref{sec:seven}, we give some examples of group actions on a closed, connected oriented $4$-manifolds to illustrate various features. For example, there is an action with zero dimensional fixed point set,  where the Borel spectral sequence does not collapse, and another with a two dimensional fixed point component where the Borel spectral sequence does not collapse.  This motivates our assumption that $H_1(F)\twoheadrightarrow H_1(M)$ is  surjective. 

In Section \ref{sec:two}, we give some general facts about the main tool we use in the proof; the Leray-Serre spectral sequence for the fibration $M\rightarrow M\times_G EG\rightarrow BG$, which is also called the Borel Spectral Sequence.  The details can be found in the books \cite{borel-seminar} and \cite{tomDieck2}. 

In Section \ref{sec:three}, we prove  the first part of Theorem A, and complete the proof in Section \ref{sec:four}.  In Section \ref{sec:five} we give some applications under the extra assumptions of homological triviality  and  $H_1(M) = 0$. In Section \ref{sec:six} we prove Theorem B.

\begin{acknowledgement}
The authors would like to thank Allan Edmonds and Michael McCooey for helpful conversations and correspondence. We would also like to thank the referee for many valuable comments on the first version of this paper.
\end{acknowledgement}

\section{The Borel Spectral Sequence}
\label{sec:two}
In this section we recall some of the standard facts about $H^*_G(M)$, where $G$ is a finite group acting on a finite dimensional $G$-CW complex $M$. In particular, these results apply to $G$-manifolds and singular cohomology with coefficients in $R=\bbZ$ (or $R=\Fp$ when indicated). The details about this construction and the spectral sequence can be found in Borel \cite{borel-seminar} and tom Dieck \cite{tomDieck2}.

The Leray-Serre spectral sequence for the fibration $M\rightarrow M\times_GEG\rightarrow BG$ is known as the Borel spectral sequence. The total space of this fibration which is known as the Borel construction, and denoted by $M_G=M\times_GEG$. The $E_2$ page of this spectral sequence is $$E^{k,l}_2(M)=H^k(G;H^l(M))$$ which converges to the cohomology $H^*(M_G)$ of the total space $M_G$. These are denoted by $H^*_G(M)$ and known as the Borel equivariant cohomology groups.  This construction is natural with respect to $G$-maps of $G$-spaces. 

In the examples in Section \ref{sec:seven}, we use Proposition \ref{prop:twotwo} given below to decide whether the Borel spectral sequence collapses. Before we state this proposition we recall some basic definitions for the convenience of the reader. In this section we will  denote the fixed set by $M^G$. 

Since $EG$ is path-connected, any  fibre inclusion $j_b\colon M\rightarrow EG\times_G M$, with $j_b(m) = (b,m)$  for $b\in EG$ and $m\in M$,  induces a well-defined map $j^*\colon H^*_G(M)\rightarrow H^*(M)$.

A \emph{cohomology extension of the fibre} is an $R$-module homomorphism of degree zero $t\colon H^*(M)\rightarrow  H^*_G(M)$ such that $j^*\circ t$ is the identity. $M$ is called \emph{totally non-homologous to zero} in $M_G$ with respect to $H^*(-)$ if $j^*$ is surjective. 

 Since a surjective map onto a free $R$-module splits, if $M$ is totally non-homologous to zero and $H^*(M)$ is a free $R$-module then a cohomology extension of the fibre exists.  Also,  if $M$ is totally non-homologous to zero in $M_G$ then $G$ acts trivially on $H^*(M)$. 
 
 One can show that \cite[Ch.~III, Prop.1.17]{tomDieck2}:  $M$ is totally non-homologous to zero in $M_G$ if and only if $G$ acts trivially on $H^*(M)$ and $E_2^{0,*}$ consists of permanent cocyles (i.e $E_2^{0,p}=E_\infty^{0,p}$). Also if we have $H^*(M)$ is finitely generated free $R$-module then \cite[Ch.~III, Prop.1.18]{tomDieck2}: $M$ is totally non-homologous to zero in $M_G$ if and only if $G$ acts trivially on $H^*(M)$ and the Borel spectral sequence collapses. In this case, $H^*_G(M)$ is a free $H^*(BG)$-module.

The \cite[Ch.~III, Prop. 4.16]{tomDieck2} comes as an application of Localization Theorem, so let us recall it briefly. 
Let $S$ be a multiplicatively closed subset of homogeneous elements in $H^*(BG)$ and $\cF(S)=\{H\leqslant G\mid S\bigcap\ker(H^*_G(G/G)\rightarrow H^*_G(G/H))\neq\emptyset\}$ then \cite[Ch.~III, Theorem 3.8]{tomDieck2}:

\begin{theorem}\label{thm:twoone}
Let $(M,A)$ be a finite dimensional relative $G$-complex.  Suppose $M\setminus A$  has finite orbit types with orbits isomorphic to $G/H$ for $H\in\cF(S)$.  Then the inclusion $A\subset M$ induces the isomorphism $S^{-1}H^*_G(M)\cong S^{-1}H^*_G(A)$.
\end{theorem}

\medskip
\noindent
\textbf{Assumption}: \emph{In the remainder of this section we list some results about the Borel cohomology $H^*_G(M)$ for finite $p$-group actions, with coefficients in $\bk:=\Fp$ understood}. 

\medskip
In this setting, the Localization Theorem has a stronger conclusion.

\begin{theorem}[{\cite[Ch.~III, Theorem 3.13]{tomDieck2}}] Let $G=(\cy p)^n$ be a $p$-torus and $M$ a finite dimensional $G$-CW complex. Then
$S^{-1}H^*_G(M)\cong S^{-1}H^*_G(M^G)$.
\end{theorem}

Let $j\colon M \to \{pt\}$ denote the map of $M$ to a point. 
\begin{corollary}\label{cor:point} Let $G=(\cy p)^n$ be a $p$-torus and $M$ a finite dimensional $G$-CW complex. Then $M^G \neq \emptyset$ if and only if $j^*\colon H_G^*(pt) \to H_G^*(M)$ is injective.
\end{corollary} 

 Here are some useful criteria for the collapse of the Borel spectral sequence: we are combining statements from
 Borel \cite[Ch.~XII, Thm 3.4]{borel-seminar} and tom Dieck
 \cite[Ch.~III, Prop.~4.16]{tomDieck2}.
\begin{proposition}[Borel]\label{prop:twotwo}Let $G = (\cy p)^n$ be a $p$-torus and $\bk = \Fp$.
Suppose the total dimension $\sum_r \dim_{\bk} H^r(M)$ is finite and $H^q(M)=0$ for $q>\dim M=n$.  Then
$$\sum_r \dim_{\bk} H^r(M^G)\leq\sum_r \dim_{\bk} H^r(M)$$
Moreover, the following are equivalent:
\begin{enumerate}\addtolength{\itemsep}{0.3\baselineskip}
\item $\sum_r \dim_{\bk} H^r(M^G)=\sum_r \dim_{\bk} H^r(M)$.
 \item $M$ is totally non-homologous to zero in $M_G$ with respect to $H^*(-)$.
\item $\dim_{\bk} H^q_G(M)=\sum_r \dim_{\bk} H^r(M)$ for $q>n$.
\item $G$ acts trivially on $H^*(M)$ and the Borel spectral sequence collapses.
\end{enumerate}
\end{proposition}

With some extra assumptions, the following statement can be  proved:
\begin{corollary}[{\cite[Ch.XII, Corollary 3.5]{borel-seminar}}]\label{cor:twofour}
Let $G$ be an elementary abelian $p$-group, and $M$ be a compact $G$-space for which $\dim_\bk M$,  $\dim_\bk H^*(M)$,  and the number of orbit types are all finite.  Assume that 
\begin{enumerate}
\item $G$ acts homologically trivially, and 
\item $H^*(M)$ is generated by elements which are transgressive in  the Borel spectral sequence. 
\end{enumerate}
 Then the fixed point set $M^G$ is non-empty if and only if the Borel spectral sequence collapses.
\end{corollary}

\section{Collapse of the Spectral Sequence}
\label{sec:three} 

Under some conditions, including the strong assumption that $H_1(F)\twoheadrightarrow H_1(M)$ is surjective, we prove the first  part of Theorem A, namely a ``collapse" result for the Borel spectral sequence.

\begin{theorem}\label{thm:threeone}
 Let $G=\cy p$ act locally linearly on a closed, connected, oriented $4$-manifold $M$, preserving the orientation,  with  fixed point set $F \neq \emptyset$.  If $H_1(F;\bZ)\twoheadrightarrow H_1(M;\bZ)$ is surjective, then the Borel spectral sequence  for $H^*_G(M)$ collapses with integral and $\Fp$ coefficients.
\end{theorem}

\begin{remark} Since all the arguments in the proof of this result are cohomological,  the conclusion should hold (at least for integral coefficients) if $\coker\{H_1(F) \to H_1(M)\} \neq 0$ is a cohomologically trivial $\ZG$-module and $H_1(M)$ is torsion free. We have not checked the details. If $H_1(M)$ has $p$-primary torsion, then the situation in this extra generality appears much more complicated.
\end{remark} 

At various points, we will need to use some properties of the group cohomology of  $G = \cy p$, Recall that the integral cohomology is a polynomial algebra $H^*(G;\bZ) = \bZ[\theta]$, where $|\theta| = 2$ is a class of degree 2. For $p$ odd, we have
$$H^*(G;\Fp) = \Fp[u]\otimes \Lambda(x)$$
where $|u| = 2$ and $|x| = 1$, with $x^2 = 0$. For $p=2$, $H^*(G;\bF_2) = \bF_2[x]$, where $|x|=1$. The cup products are natural with respect to the change of coefficients
$\bZ \to \Fp$, and the induced maps $H^{2k}(G;\bZ) \to H^{2k}(G; \Fp)$ are isomorphisms for $k >0$ and surjective for $k=0$.  The differentials in the  $E_r$ terms of the Borel spectral sequence for $H^*_G(M)$ are multiplicative with respect to cup products in the cohomology of $M$ and $G$.

\medskip
Before starting the proof of Theorem \ref{thm:threeone} we will collect some useful remarks:
 \begin{enumerate}
\item Since  $G$ preserves the orientation on $M$ (automatic for $p$ odd), and $H_1(F) \to H_1(M)$ is surjective, $G$ acts trivially on the homology and cohomology of $M$, except possibly for $H_2(M) \cong H^2(M)$. 
\item Let $A \subset F$ denote  a non-empty $1$-dimensional subset of the fixed point set,  such that the map 
$H_1(A)\twoheadrightarrow H_1(M)$ is  surjective. For example, take $A$ to be a $1$-skeleton of $F$.
\item The induced map $H^1(M) \to H^1(A)$ is injective.
\item By applying duality to a neighhourhood of $A$ in $M$,  we have $H^*(M\mbox{-}A) \cong H_{4-*}(M,A)$, and similarly $H^*(M,A) \cong H_{4-*}(M\mbox{-}A)$.
\item The statements  so far also hold for homology and cohomology with $\Fp$-coefficients.
\item  $H_1(M,A) = \ker\{H_0(A) \to H_0(M)\}$ is $\bZ$-torsion free, with trivial $G$-action.
\item  $H^2(M,A)$ is $\bZ$-torsion free: its torsion subgroup is $\Ext(H_1(M,A), \bZ)=0$.
\end{enumerate}
\begin{proof}
We first consider the $E_2$-page of the Borel spectral sequence $E_2^{k,l}(M)=H^k(G;H^l(M))$ and show that $d_2$ differentials are zero. Integral coefficients are understood unless $\Fp$ coefficients are stated explicitly.

\subsection{The maps \boldmath{$d_2^{k,4}\colon E^{k,4}_2(M)\rightarrow E^{k+2,3}_2(M)$}:} For any fixed point  $x\in F$, the inclusion map $i\colon M\mbox{-}\{x\}\hookrightarrow M$ induces a homomorphism $i^*\colon H^n(M)\rightarrow H^n(M\mbox{-}\{x\})$ which is zero for $n\geq 4$ and isomorphism for other dimensions.  The corresponding map of spectral sequences  $E^{k,l}_r(M)\rightarrow E^{k,l}_r(M\mbox{-}\{x\})$  is trivial when $l=4$ and an isomorphism otherwise.  By naturality we have the commutative diagrams of differentials;
$$\xymatrix{& H^k(G;H^4(M))\ar[d]_{i^*}\ar[r]^{d_2^{k,4}}& H^{k+2}(G;H^3(M))\ar[d]^{\cong}\\
& H^k(G;H^4(M\mbox{-}\{x\}))\ar[r]^{d_2^{k,4}}& H^{k+2}(G;H^3(M\mbox{-}\{x\})) }$$
Since $H^4(M\mbox{-}\{x\})=0$, we have $i^*=0$ and $H^3(M\mbox{-}\{x\})\cong H^3(M)$,  so $d_2^{k,4}=0$. The same argument works for $\Fp$ coefficients.

\subsection{The maps \boldmath{$d_2^{k,1}\colon E^{k,1}_2(M)\rightarrow E^{k+2,0}_2(M)$}:} Similarly, for any $x\in F$ consider the map $j^*\colon H^n(M, \{x\})\rightarrow H^n(M)$ induced by $j\colon (M,\emptyset)\rightarrow (M,\{x\})$. From the long exact sequence in relative cohomology, $j^*$ is isomorphism for all $n\geq 1$.   The corresponding map of spectral sequences  $E^{k,l}_r(M,\{x\})\rightarrow E^{k,l}_r(M)$ is also isomorphism for $l\geq 1$.  By naturality, we again have the commutative diagrams of differentials; 
$$\xymatrix{& H^k(G;H^1(M,\{x\}))\ar[d]_{j^*}\ar[r]^{d_2^{k,1}}& H^{k+2}(G;H^0(M,\{x\}))\ar[d]^{}\\
& H^k(G;H^1(M))\ar[r]^{d_2^{k,1}}& H^{k+2}(G;H^0(M)) }$$
Since $H^0(M,\{x\})=0$ and $j^*$ is isomorphism then $d_2^{k,1}=0$. 
The same argument works for $\Fp$ coefficients.

\subsection{The maps \boldmath{$d_2^{k,3}\colon E^{k,3}_2(M)\rightarrow E^{k+2,2}_2(M)$}:} 
 From the long exact sequence of relative homology, and $H_2(A)=0$, we get injectivitiy of $H_2(M)\rightarrowtail H_2(M,A)$. 
 Since  $H_1(A)\twoheadrightarrow H_1(M)$ is surjective, we conclude that the map $H_1(M)\rightarrow H_1(M,A)$ is zero. By duality, the map $H^3(M)\rightarrow H^3(M\mbox{-}A)$ is zero.  
 This also  holds for $\Fp$ coefficients. We obtain the commutative diagram
\eqncount
\begin{equation}\label{eq:diag1}
\vcenter{\xymatrix{&0\ar[r]& H^2(M)\ar[r]\ar[d]_{\cong}& H^2(M\mbox{-}A)\ar[d]_{\cong}\ar[r]& K\ar[r]\ar@{ >->}[d]&0\\
&0\ar[r]& H_2(M)\ar@{ >->}[r]& H_2(M,A)\ar[r]& H_1(A)\ar@{->>}[r]& H_1(M)\ar[r]& 0}}
\end{equation}
where $K := \ker\{H_1(A) \to H_1(M)\}$. 
For  $k\geq 0$ even,  $H^{k+1}(G;K)=0$  since $K$ is  $\bZ$-torsion free with trivial $G$-action.  
When we apply group cohomology to the upper short exact sequence in \eqref{eq:diag1} we get the long exact sequence 
$$\xymatrix{\cdots \ar[r]&H^{k+1}(G;K)\ar[r]& H^{k+2}(G;H^2(M))\ar[r]& H^{k+2}(G;H^2(M\mbox{-}A))\ar[r]&\cdots}$$
 It follows that the map $H^{k+2}(G;H^2(M))\rightarrowtail H^{k+2}(G;H^2(M\mbox{-}A))$ is injective for $k$ even.

Since the map $H^3(M)\rightarrow H^3(M\mbox{-}A)$ is zero, the induced map  in group cohomology $H^{k}(G;H^3(M))\rightarrow H^{k}(G;H^3(M\mbox{-}A))$ is also zero. By naturality of spectral sequences with respect to the inclusion $M\mbox{-}A \hookrightarrow M$ we have the following commutative diagram
$$\xymatrix{& H^k(G;H^3(M))\ar[d]_{0}\ar[r]^{d_2^{k,3}}& H^{k+2}(G;H^2(M))\ar@{ >->}[d]\\
& H^k(G;H^3(M\mbox{-}A))\ar[r]^{d_2^{k,3}}& H^{k+2}(G;H^2(M\mbox{-}A)) }$$
implying $d_2^{k,3}=0$ for $k$ even. For $\Fp$ we are missing the injectivity of the right-hand vertical map. However, the isomorphism $H^3(M) \otimes \Fp \cong H^3(M;\Fp)$ implies that $H^0(G; H^3(M)) \to H^0(G; H^3(M;\Fp))$ is surjective, since both coefficients have trivial $G$-action, so reduces to the surjection $H^3(M) \to H^3(M;\Fp)$. Then naturality gives $d_2^{2i, 3} = 0$ for integral or $\Fp$ coefficients.

\medskip
\noindent
\textbf{The differentials $d_2^{k, 3}$  for $k$ odd ($\Fp$-coefficients)}:
If $H_1(M)$ has no $p$-torsion, then $ H^3(M;\Fp)=0$ and $d_2^{k, 3}=0$ with $\Fp$-coefficients  for $k$ odd.
To handle the case where $H^3(M)$ has $p$-primary torsion and $k$ is odd, we compare with the  $\Fp$-coefficient spectral sequence via the change of coefficients $\bZ \to \Fp$.
 Note that since $H^0(G; H^3(M)) \to H^0(G; H^3(M;\Fp))$ is surjective, and $d_2^{0,3} = 0$, we see that 
$$d_2^{0,3} \colon H^0(G; H^3(M;\Fp)) \to H^2(G; H^2(M;\Fp))$$
is also zero. Now we use the multiplicativity of the $\Fp$-coefficients spectral sequence, and the fact that
$$ \cup\, x \colon H^0(G; H^3(M;\Fp)) \to H^1(G; H^3(M;\Fp))$$
is surjective (since the coefficients have trivial $G$-action), where $0 \neq x \in H^1(G; \Fp)$,  to conclude that 
$$d_2^{1,3} \colon H^1(G; H^3(M;\Fp)) \to H^3(G; H^2(M;\Fp))$$
is zero for $\Fp$ coefficients,  and hence for all odd $k$ by naturality and periodicity. 
We have now shown that $d_2^{k,3}=0$ for all $k$ in the spectral sequence with $\Fp$ coefficients . 

\begin{remark} If $H_1(M)\cong H^3(M)$ is $p$-primary torsion free, then  $H^k(G,H^3(M))=0$ for $k$ odd,  since $H_1(M)\cong H^3(M)$ is a trivial $G$-module, due to the  assumption that $H_1(F)\twoheadrightarrow H_1(M)$ is surjective. Hence the differentials $d_2^{k,3}=0$ with integral coefficients for all odd $k$,  if the order of $H_1(M)$ is not divisible by $p$.
\end{remark}

We will return to   the remaining differentials \textbf{$d_2^{k,3}$,  for  $k$ odd and integral coefficients},  after showing that the spectral sequence collapses for $\Fp$ coefficients.

\subsection{The maps \boldmath{$d_2^{k,2}\colon E^{k,2}_2(M)\rightarrow E^{k+2,1}_2(M)$}:} 
Let $T = \ker\{H^2(M,A) \to H^2(M)\}=\coker\{H^1(M) \to H^1(A)\}$. Since $T$ is a quotient of $H^1(A)$, 
it has trivial $G$-action. Since $ T\subseteq H^2(M,A)$, it  is $\bZ$-torsion free.
Since $H^2(A)=0$, we have 
a short exact sequence:
$$\xymatrix{0\ar[r]& T\ar[r]& H^2(M,A)\ar@{->>}[r]^{\alpha}& H^2(M)\ar[r]&0}$$
which induces a long exact sequence in group cohomology:  
$$\xymatrix{\cdots \ar[r]&H^{k}(G;T)\ar[r]& H^{k}(G;H^2(M,A))\ar[r]& H^{k}(G;H^2(M))\ar[r]&H^{k+1}(G;T)\ar[r]&\cdots}$$ 
Therefore  the map $H^{k}(G;H^2(M,A))\rightarrow H^{k}(G;H^2(M))$ is surjective for $k$ even, since $H^{k+1}(G;T)=0$ in this case.

By naturality with respect to the map of pairs $(M,\emptyset)  \to (M,A)$  we have the commutative diagram: 
$$\xymatrix{& H^k(G;H^2(M,A))\ar@{->>}[d]\ar[r]^{d_2^{k,2}}& H^{k+2}(G;H^1(M,A))\ar[d]_{0}\\
& H^k(G;H^2(M))\ar[r]^{d_2^{k,2}}& H^{k+2}(G;H^1(M)) }$$
We note that the map $H^1(M,A)\rightarrow H^1(M)$ is zero, since $H^1(M) \to H^1(A)$ is injective. 
Hence the map $H^{k+2}(G;H^1(M,A))\rightarrow H^{k+2}(G;H^1(M))$ is zero, for $k$ even.
and $d_2^{k,2}=0$, for $k$ even.  For odd $k$, we  have $H^{k+2}(G;H^1(M))=0$ since $H^1(M)$ is torsion free with trivial $G$-action, and $d_2^{k,2}=0$ also for $k$ odd (with integral coefficients).

\medskip
To understand the $d_2^{k,2}$ differentials with $\Fp$ coefficients we use the multiplicative structure in the spectral sequence. Suppose that $0 \neq  d_2^{0,2}(z) \in E_2^{2,1} = H^2(G; H^1(M;\Fp))$, for some $z \in E_2^{0,2}$. Since the cup product pairing 
$$H^2(G; H^1(M;\Fp)) \times H^2(G; H^3(M;\Fp)) \to H^4(G; H^4(M;\Fp)) = \Fp$$
is non-singular, there exists $w \in H^2(G; H^3(M;\Fp)) = E_2^{2,3}$ such that $d_2^{0,2}(z) \cdot w \neq 0$. But 
$$0 = d_2^{2,5}(z\cdot w)  = z \cdot d_2^{2,3}(w) -  d_2^{0,2}(z) \cdot w$$
 since $z\cdot w \in E_2^{2,5} = 0$ and $ d_2^{2,3}(w) = 0$, as shown above. This is a contradiction, and hence $d_2^{0,2} = 0$.
 Since $\cup\, x\colon H^2(G; H^1(M;\Fp)) \cong H^3(G; H^1(M;\Fp))$, we have
$ d_2^{1,2} = 0$. This completes the proof that all the $d_2$ differentials are zero for $\Fp$ coefficients.

\subsection{Vanishing of differentials in the $E_3$-page:}   Obviously $d_3^{k,1}=0$, and we again use the  maps induced from  $i\colon M\mbox{-}\{x\}\hookrightarrow M$  to show $d_3^{k,4}=0$, and $j\colon (M,\emptyset) \to (M, \{x\})$ to show $d_3^{k,2}=0$. 
$$\xymatrix@C-4pt{&E^{k,4}_3(M) =  H^k(G;H^4(M))\ar[d]_{i^*}\ar[r]^{d_3^{k,4}}& H^{k+3}(G;H^2(M))/\image d_2^{k+1,3}(M)\ar[d]^{\cong}\\
&E^{k,4}_3(M\mbox{-}\{x\}) = H^k(G;H^4(M\mbox{-}\{x\}))\ar[r]^{d_3^{k,4}}&
 H^{k+3}(G;H^2(M\mbox{-}\{x\}))/\image d_2^{k+1,3}(M\mbox{-}\{x\}) }$$
We have  $H^2(M\mbox{-}\{x\})\cong H^2(M)$ and  $H^3(M\mbox{-}\{x\})\cong H^3(M)$, which gives the righthand vertical isomorphism.
The map  $i^*=0$  since $H^4(M\mbox{-}\{x\})=0$,  so $d_3^{k,4}=0$. 

\medskip
Next consider the diagram:
$$\xymatrix@C+8pt{& H^k(G;H^2(M,\{x\}))/\image d^{k-2,3}_2(M,\{x\}) \ar[d]_{j^*}\ar[r]^(0.6){d_3^{k,2}}& H^{k+3}(G;H^0(M,\{x\}))\ar[d]^{}\\
& H^k(G;H^2(M))/\image d^{k-2,3}_2(M)\ar[r]^(0.6){d_3^{k,2}}& H^{k+3}(G;H^0(M)) }$$
Since $H^0(M,\{x\})=0$ and $j^*$ is  an isomorphism,  we have $d_3^{k,2}=0$.
The same arguments work for $\Fp$ coefficients.

For  $d_3^{k,3}$ we again use naturality and the following commutative diagram;
$$\xymatrix{& H^k(G;H^3(M))\supseteq \ker d_2^{k,3}(M)\ar[d]_{i^*}\ar[r]^(0.6){d_3^{k,3}}& H^{k+3}(G;H^1(M))\ar[d]_{i^*}^{\cong}\\
& H^k(G;H^3(M\mbox{-}A))\supseteq \ker d_2^{k,3}(M\mbox{-}A)\ar[r]^(0.6){d_3^{k,3}}& H^{k+3}(G;H^1(M\mbox{-}A)) }$$
By duality,  $H^1(M) \to H^1(M\mbox{-}A)$ is an isomorphism, and  
so is the map 
$$i^*\colon H^{k+3}(G;H^1(M))\rightarrow H^{k+3}(G;H^1(M\mbox{-}A)).$$
 Since  the map $H^3(M)\to H^3(M\mbox{-}A)$ is zero (as noted above),  we have $d_3^{k,3}=0$.  
The same arguments work for $\Fp$ coefficients. For integral coefficients (where we have not yet shown $\ker d_2^{k,3}=0$ if $k$ is odd),  we are using the vanishing of $d_2^{k-2,4}$ to see that the domain 
of $d_3^{k,3}$ is $\ker d_2^{k,3} \subseteq H^k(G;H^3(M))$.

\subsection{Vanishing of differentials in the $E_4$-page:}   Obviously $d_4^{k,1}=0$ and $d_4^{k,2}=0$, and  again we use the induced maps $ i^* $ to show $d_4^{k,4}=0$ and $j^*$ to show $d_4^{k,3}=0$. 
$$\xymatrix{& H^k(G;H^4(M))\ar[d]_{i^*}\ar[r]^{d_4^{k,4}}& H^{k+4}(G;H^1(M))\ar[d]^{\cong}\\
& H^k(G;H^4(M\mbox{-}\{x\}))\ar[r]^{d_4^{k,4}}& H^{k+4}(G;H^1(M\mbox{-}\{x\})) }$$
Since $H^4(M\mbox{-}\{x\})=0$, we have $i^*=0$ and $H^1(M\mbox{-}\{x\})\cong H^1(M)$,  so $d_4^{k,4}=0$.  
$$\xymatrix{& H^k(G;H^3(M,\{x\}))\supseteq \ker d_2^{k,3}(M,\{x\})\ar[d]_{j^*}\ar[r]^(0.6){d_4^{k,3}}& H^{k+4}(G;H^0(M,\{x\}))\ar[d]^{}\\
& H^k(G;H^3(M))\supseteq \ker d_2^{k,3}(M)\ar[r]^(0.6){d_4^{k,3}}& H^{k+4}(G;H^0(M)) }$$
Since $H^0(M,\{x\})=0$ and $j^*$ is an isomorphism, it follows that $d_4^{k,3}=0$. 
The same arguments work for $\Fp$ coefficients. 

\subsection{Vanishing of differentials in the $E_5$-page} There is only one differential to consider 
$d_5^{k,4}\colon E^{k,4}_5(M)\rightarrow E^{k+5,0}_5(M)$ which can easily 
shown to be zero by again using $j^*$:
$$\xymatrix{& H^k(G;H^4(M,\{x\}))\ar[d]_{j^*}\ar[r]^{d_5^{k,4}}& H^{k+5}(G;H^0(M,\{x\}))\ar[d]^{}\\
& H^k(G;H^4(M))\ar[r]^{d_5^{k,4}}& H^{k+5}(G;H^0(M)) }$$
Since $H^0(M,\{x\})=0$ and $j^*$ is isomorphism then $d_5^{k,4}=0$. 
The same arguments work for $\Fp$ coefficients.  

\begin{remark} For the vanishing of the differentials $d_r^{k,r-1}$  hitting the $(*,0)$ line, we could just have cited 
 Corollary \ref{cor:point},  since $F \neq \emptyset$.
\end{remark}

\medskip
We have now shown that the Borel spectral sequence with $\Fp$ coefficients collapses, and that $E_3 = E_\infty$ with integral coefficients (independently of the vanishing of  $d_2^{2i+1,3}$). 

\subsection{The maps \boldmath{$d_2^{2i+1,3}\colon E^{2i+1,3}_2(M;\bZ)\rightarrow E^{2i+3,2}_2(M;\bZ)$}:} 
We will show that the differentials
$d_2^{2i+1,3} =0$ by comparing the integral calculations with the mod $p$ calculations. 

Note that the groups $H_G^q(M)$ with integral coefficients are all $\bk$-vector spaces for $q >4$ (with notation $\bk := \Fp$ as before). This follows from the isomorphism $H_G^q(M) \cong H_G^q(F) \cong H^q(F \times BG)$ (see \cite[Proposition 2.1]{Edmonds:1998}).

There is a  short exact sequence of $\bk$-vector spaces
\eqncount
\begin{equation}\label{eq:dima}
0 \to H^5_G(M) \to H^5_G(M; \Fp) \to H^6_G(M) \to 0.
\end{equation}
and we will compute both sides of the  resulting equality 
\eqncount
\begin{equation}\label{eq:dimb}
 \dim_\bk H^5_G(M)  + \dim_\bk H^6_G(M)=  \dim_\bk H^5_G(M; \Fp) 
 \end{equation}
via separate calculations. 

Suppose that  $d_2^{1,3} \neq 0$,  and  we let 
$b = \dim_\bk (\image d_2^{1,3}) =  \dim_\bk (\image d_2^{3,3}) $ (by periodicity). 
Let $t =  \dim_\bk H^{odd}(G;H^3(M))$ and note that $H^{odd}(G; H^1(M)) = 0$.
Since $H^3(M) \cong H_1(M)$ has trivial $G$-action, 
we have $\dim_\bk H^{2}(G; H^3(M)) = b_1(M) + t$.

\begin{lemma}\label{lem:dimc}
\mbox{}
\begin{enumerate}
\item  $\dim_\bk H^5_G(M) = 
2b_1(M) +  (t-b)  +  \dim_\bk H^3(G; H^2(M))$;
\item $\dim_\bk H^6_G(M) = 
2 +  (t-b)  +  \dim_\bk H^4(G; H^2(M)) $;
\item  $\dim_\bk H^5_G(M;\Fp) = 
2 + 2b_1(M) +  2t  +  \dim_\bk H^3(G; H^2(M;\Fp)) $.
\end{enumerate}
\end{lemma}

\begin{proof}
We compute:
$$ \dim_\bk H^5_G(M) = \sum_{i = 0}^4  \dim_\bk H^{5-i}(G; H^i(M)) -b$$
 and note that 
$$\dim_\bk E_\infty^{3.2} = \dim_\bk  H^3(G; H^2(M)) -b.$$
We then obtain the first formula after taking into account the vanishing of all the other differentials.
Similarly, 
$$ \dim_\bk H^6_G(M) = \sum_{i = 0}^4  \dim_\bk H^{6-i}(G; H^i(M)) -b$$
after substituting the value
$$\dim_\bk E_\infty^{3.3} = \dim_\bk  H^3(G; H^3(M)) -b = t-b,$$
and we obtain the second formula.
To compute the third formula, we note that $$\dim_\bk H^1(G; H^4(M;\Fp))= \dim_\bk H^5(G; \Fp) =1,$$
 and 
$$\dim_\bk H^2(G; H^3(M;\Fp)) = \dim_\bk H^4(G; H^1(M;\Fp)) = t + b_1(M).$$
\end{proof}

In order to compare the integral and mod $p$ formulas, we need some information 
about the structure of $H^2(M;\Fp)$ as a $G$-module. We can decompose 
$$H^2(M)/\Tors \cong \bZ^{r_0(M)} \oplus \bZ[\zeta_p]^{r_1(M)} \oplus \Lambda^{r_2(M)}$$
as a $G$-module (this uses the classification of $\bZ$-torsion free $\ZG$-modules, 
and an argument with $G_0(\ZG)$ due to Swan). This module supports a non-singular $G$-invariant symmetric bilinear form arising from the intersection form on $M$.

Let $T = \Tors (H^2(M))$  and note that $T^* = \Ext^1(T, \bZ) = \Tors (H^3(M)) \cong \Tors (H_1(M))$. In our case, $T^* \cong T$ as $G$-modules with trivial $G$-action. We introduce the notation 
$V := T \otimes \Fp$ and $V^* := \mbox{}_pT \cong \Hom_\bk(V, \Fp)$ for the elements of exponent $p$ in $T$.

\begin{definition}\label{def:splittype} Let $V := \Tors(H_1(M)) \otimes \Fp$ and $V^* = \Hom_\bk(V, \Fp)$. We say that the $G$-representation $H^2(M;\Fp)$ has \emph{split type} if the short exact sequences:
$$0 \to H^2(M)\otimes \Fp \to H^2(M;\Fp) \to V^* \to 0$$
and
$$0 \to V \to H^2(M)\otimes\Fp \to (H^2(M)/\Tors)\otimes \Fp \to 0$$
of  $G$-representations are \emph{split exact over $G$}.
\end{definition}

 We will show that this condition is always satisfied in the setting of Theorem A.
 
\begin{lemma}\label{lem:splittype}
 If $G = \cy p$ 
 then $H^2(M;\Fp)$ has split type as a $G$-representation, and 
$H^2(M;\Fp) \cong (H^2(M)/\Tors) \otimes \Fp \oplus V \oplus V^*$ as a $G$-module. 
\end{lemma}

\begin{proof} We have a short exact sequence of $\bk$-vector spaces with $G$-action:
$$0 \to H^2(M)\otimes \Fp \to H^2(M; \Fp) \to V^*\to 0$$
where $V^* \cong \mbox{}_p(H^3(M))$ as a trivial $G$-representation. Let $\bar L := (H^2(M)/\Tors) \otimes \Fp$, and consider the short exact sequence
$$0 \to V \to H^2(M)\otimes \Fp \to \bar L \to 0.$$
Since $\bar L$ supports a non-degenerate $G$-invariant symmetric bilinear form $\bar L \times \bar L \to \Fp$ (induced by the intersection from of $M$), it follows that this sequence splits over $G$ and we have $H^2(M)\otimes \Fp \cong V \oplus \bar L $ as $G$-modules. Similarly, the submodule $\bar L$ of $H^2(M;\Fp) $ is a direct summand, and we have a splitting
$$ H^2(M;\Fp) \cong \bar L \oplus H(V),$$
where $H(V)$ is determined by an extension 
$0 \to V \to H(V) \to V^*\to 0$. The $G$-module is an $\Fp$-vector space, with isometry $t$ given by a generator of $G = \la t \ra$. 

To show that the extension determining $H(V) \subseteq  H^2(M;\Fp)$  is $G$-split, consider the diagram
$$\xymatrix{0 \ar[r]& H^2(M)\otimes \Fp \ar[r]\ar[d]&H^2(M; \Fp) \ar[r]\ar[d] & V^*\to 0\\
& H^2(M\mbox{-}A)\otimes \Fp \ar[r]^{\cong}&H^2(M\mbox{-}A; \Fp)&
}$$
The lower isomorphism comes from the Bockstein sequence for $M\mbox{-}A$, and the fact that  $H^3(M\mbox{-}A)$ is $\bZ$-torsion free.
The  short exact sequence  in diagram \eqref{eq:diag1}:
$$0 \to H^2(M) \to H^2(M\mbox{-}A) \to K \to 0$$
shows that  $\Tors H^2(M) \xrightarrow{\cong} \Tors  H^2(M\mbox{-}A)$, 
 since $K$ is $\bZ$-torsion free. After tensoring  with $\Fp$,  we obtain a $G$-splitting of the submodule $H(V)$. 
\end{proof}

\begin{corollary}
$\dim_\bk H^3(G; H^2(M;\Fp)) = 2t + r_0(M) + r_1(M)$.
\end{corollary}
\begin{proof} This follows from Lemma \ref{lem:splittype}.
\end{proof}
We can put this information together with the formulas in Lemma \ref{lem:dimc}. We have
$$\dim_\bk H^3(G; H^2(M)) = t + r_1(M), \qquad \dim_\bk H^4(G; H^2(M)) = t+ r_0(M).$$
By substituting the values obtained into the dimension formula \eqref{eq:dimb}, we conclude that $b = \dim_\bk (\image d_2^{1,3}) = 0$.
In other words, we have shown that the differential $d_2^{k, 3} = 0$ for $k$ odd in the  Borel spectral sequence with integral coefficients.  This completes the proof of Theorem \ref{thm:threeone}, and establishes the first part of Theorem A..
\end{proof}

\begin{remark}
Since $H^r_G(M) \cong H^r_G(F) = H^r(F \times BG)$ for $r >4$,  we have
$\ \dim_\bk H^5_G(M) =  \dim_\bk H^5_G(F) =  b_1(F)$ and $ \dim_\bk H^6_G(F) = b_0(F)+ b_2(F)$.  
In addition, $ \dim_\bk H^5_G(F) =  \dim_\bk H^3_G(F)$ since $F$ consists of surfaces and isolated points
By computing the trace of the action of a generator on $H^*(M)$, we  obtain the relation
$$\chi(F)  = b_2(F) - b_1(F) + b_0(F) = 2 -2b_1(M) + r_0(M) - r_1(M).
$$
However, since  $H_G^q(M) = H_G^q(F)$ for $q >4$, we can use
the Herbrand quotient formula
$$ \dim H^4((G; H^2(M)) - \dim H^3((G; H^2(M)) = r_0(M) - r_1(M)$$
and further calculations similar to those above for $H_G^q(M)$, to show directly that
$\chi(F) = \dim H^6_G(M) - \dim H^5_G(M)$.
\end{remark}

\section{A non-collapse result}\label{sec:four}
We complete the proof of Theorem A by showing that our  surjectivity condition  for the map 
$H_1(F) \to H_1(M)$ is necessary in many cases (see Section \ref{sec:seven} for some examples).

\begin{proposition}\label{prop:fourone}
Suppose that $G=\cy p$ acts locally linearly on a closed, connected, oriented $4$-manifold $M$, preserving the orientation,  with non-empty fixed point set $F$. If $$\ker(H^1(M;\bZ) \to H^1(F;\bZ))$$ is non-trivial, but has trivial $G$-action, then the Borel spectral sequence with integral coefficients does not collapse.
\end{proposition}

\begin{proof} The proof uses the fact that $H_G^q(M, F) = 0$ for $q > 4$, implying that $E_\infty^{q,1}(M,F) = 0$  for $q >4$  in the Borel spectral sequence. We will show that this leads to a contradiction.

If  $H^1(M)\rightarrow H^1(F)$ is not injective, we let $0 \neq K=\ker(H^1(M)\rightarrow H^1(F))$, which by assumption  has trivial $G$-action. Therefore $H^{2r}(G; K) \neq 0$. 
We consider the relative long exact sequence for the pair $(M,F)$, we get short exact sequences
$$ 0 \to H^0(F)/H^0(M) \to  H^1(M,F) \to K \to 0$$
and
$$ 0 \to K \to H^1(M) \to L \to 0$$
where $L = \image (H^1(M) \to H^1(F))$. Since both  $L$ and $H^0(F)/H^0(M)$ are $\bZ$-torsion free with trivial $G$ action, we have $H^{2r-1}(G; L ) = 0$ and $H^{2r+1}(G; H^0(F)/H^0(M) ) = 0$. By applying group cohomology to the sequences above, we obtain
$$\xymatrix@R-5pt@C-5pt{ H^{2r}(G;H^1(M,F))\ar[rr]^-{\alpha}\ar[dr]|!{[rr];[dr]}\hole && H^{2r}(G;H^{1}(M))\\
&H^{2r}(G;K)\ar[ur] }$$
where $ H^{2r}(G;H^{1}(M,F))\twoheadrightarrow H^{2r}(G;K) $ is surjective, and $ H^{2r}(G;K)\rightarrowtail H^{2r}(G;H^{1}(M)) $ is injective. Since  $E_\infty^{2r,1}(M,F) = 0$  for $2r >4$, some differential must hit a pre-image of a non-zero element in 
$H^{2r}(G;K)$. By comparison, we see that the Borel spectral sequence for $H_G^*(M)$ has  a non-zero differential, and hence does not collapae. \end{proof}
 
\section{Homologically trivial actions}\label{sec:five}

We will first consider the Borel spectral sequence for a cyclic $p$-group acting homologically trivially. We use $\Fp$ coefficients throughout this section.

\begin{proposition}\label{prop:cyclic}
 Let $G = \cy p$ act  homologically trivially on $M$, and assume that $\chi(M) \neq 0$ and the fixed set $F$ is discrete. Then the differentials $d_r = 0$, for $r \geq 3$,  in the Borel spectral sequence with $\Fp$-coefficients. Moreover, $b_2(M) \geq 2 b_1(M)$ and the Borel spectral sequence does not collapse unless $b_1(M) = 0$.
 In particular, $d_2^{k,3}$ is injective for  $k\geq 0$ and $d_2^{k,2}$ is surjective for $k \geq 0$.
\end{proposition}
\begin{proof} The difference in the multiplicative structure of the $\Fp$-cohomology algebras 
of $G= \cy p$ for $p$ odd and $p=2$ does not affect the proof, so we consider both cases together. 
 Since the action is homologically trivial and $\chi(M) \neq 0$, the fixed set $F \neq \emptyset$, and $F$  consists of $\chi(M)$ isolated points.

We will prove the result by computing the dimension of $H^5_G(M)$ which must be equal to $\dim H^5_G(F) = \chi(M)$ by \cite[Proposition 2.1]{Edmonds:1998}.

The same arguments used in the proof of Theorem A show that the differentials $d_2^{0,4}$ and $d_2^{0,1}$ are both zero (these work the same way for $p=2$ as for $p$ odd). Moreover, since $F \neq \emptyset$ the inclusion induces an injection $H^*(G) \to H^*(M_G)$, so $E_2^{*,0} = E_\infty^{*,0}$.

The key to understanding the other $d_2$ differentials is the result of Sikora \cite[Section 3.3]{Sikora:2004}, which shows that $E_3^{2,1} \cong E_3^{2,3}$ by recognizing a Poincar\'e duality structure on certain terms of the Borel spectral sequence. Since $E_2^{2,1} \cong E_2^{2,3}$, and $d_2^{2,1} = 0$, it follows that $\ker d_2^{2,3} \cong \coker d_2^{0,2}$. Therefore, if $R = \dim \ker d_2^{2,3}$, we have $\dim \image d_2^{0,2} = b_1(M) -R$. Therefore
$$\dim \image d_2^{2,3} = \dim \image d_2^{1,3} = \dim \image d_2^{3,2} =b_1(M) -R.$$
so we have $\dim E_3^{2,3} = \dim E_3^{4,1} = R$, and $\dim E_3^{3,2} = b_2(M) - 2 b_1(M) + 2R$.
 A detailed study of the possible $d_3$ differentials, now shows that from the relation
$$\sum \dim E_3^{k, 5-k} = 2 + b_2(M) - 2b_1(M) + 4R$$
and the convergence to $H^5_G(M) \cong H^5_G(F)$, we must have $R =0$.  The details are similar to those in Section \ref{sec:three}. Since $F\neq \emptyset$, we conclude that $d_r^{*,4}=0$ for $r = 3, 4, 5$.  The only remaining differential to consider is $d_3^{*,3}$, but since Poincar\'e duality is preserved between $E_4^{4,1} \cong E_4^{4,3}$, we see that $d_3^{*,3} = 0$. 
Therefore $d_2^{k,3}$ is injective for  $k\geq 0$ and $d_2^{k,2}$ is surjective for $k \geq 0$. By the dimension count above, this shows that the higher differentials $d_r = 0$ for $r \geq 3$.
\end{proof}

With extra assumptions such as homological triviality and torsion free $H_1(M)$, we can prove the converse of  Theorem \ref{thm:threeone}. 

\begin{corollary}\label{cor:sevenone}
Let $G=\cy p$ for $p$ odd act locally linearly, homologically trivially on a closed, connected, oriented $4$-manifold $M$ with the fixed point set $F$ non-empty and $H_1(M;\bZ)$ torsion-free. Then the Borel spectral sequence with integral coefficients collapses if and only if $H_1(F)\twoheadrightarrow H_1(M)$ is surjective. 
\end{corollary}

\begin{proof} Since we are assuming that $H_1(M;\bZ)$ torsion-free, the condition that $H_1(F) \to H_1(M)$ is surjective is equivalent to the condition that $H^1(M;\bZ) \to H^1(F;\bZ)$ is injective. The result now follows from Theorem \ref{thm:threeone} and Proposition \ref{prop:fourone}.
\end{proof}

\begin{corollary}\label{cor:threefour}
Let $p$ be an odd prime.  If $G=\cy p$ acts homologically trivially and locally linearly on $M$ with $\chi(M)\neq 0$, such that   $H_1(F)\twoheadrightarrow H_1(M)$ is surjective, then the $\Fp$-Betti numbers satisfy $b_1(F)=2b_1(M)$ and $b_0(F)+b_2(F)=2+b_2(M)$.
\end{corollary}

\begin{proof}
Since the action is homologically trivial, $\chi(F)=\chi(M)\neq 0$  by the Lefschetz fixed point theorem  and hence $F\neq \emptyset$. By  Theorem \ref{thm:threeone} we know that Borel spectral sequence collapses and  by Proposition \ref{prop:twotwo} (with $\bk = \Fp$ coefficients) we have 
$$\sum_r \dim_{\bk} H^r(F)=\sum_r \dim_{\bk}H^r(M).$$
 It follows that $b_1(F)=2b_1(M)$ and $b_0(F)+b_2(F)=2+b_2(M)$ for odd $p$.
 \end{proof}

We can also apply our results to  some actions of rank two groups (compare \cite[Proposition 6.1]{Edmonds:1998}).

\begin{remark}\label{rem:seventwo} If $G$ acts homologically trivially and the Borel spectral sequence $E(M_K)$ does not collapse for the subgroup $K\leq G$ of a group $G$ then $E(M_G)$ does not collapse. 
 \end{remark}

\begin{proposition}\label{prop:threefive}
Let $p$ be an odd prime.  If $G=\cy p \times \cy p$ acts homologically trivially, locally linearly on $M$ with non-empty fixed point set.  
Suppose that $H_1(M;\bZ)$ is torsion free. Then the Borel spectral sequence with $\Fp$ coefficients collapses if and only if $ H_1(M) =0$.
\end{proposition}

\begin{proof} 
Suppose that the fixed set $F$ contains a $2$-dimensional component $F_1 \subseteq F$. Consider the action of $G$ on the boundary of an $G$-equivariant  normal $2$-disk neighbourhood of  a point $ x\in F_1$.  Since $G = \cy p \times \cy p$ and $p$ is odd, this gives a contradiction since there is no such $G$-action on a circle. Hence the fixed set $F$ consists of a finite set of isolated points.

 Next we remark that in a small $G$-invariant neighbourhood $U$ of each fixed point $x \in F$ has $T_xU \cong V_1 \oplus V_2$, where $V_i = \Fix (T_xU, K_i)$, for two order $p$ subgroups $K_1 =\la a\ra$ and $ K_2=\la b\ra$  of  $G$ which have $K_1 \cap K_2 = \{1\}$. 

 Therefore each 
 $G$-fixed point $x \in F$ is contained in exactly two singular surfaces $S_1$ and $S_2$, where $S_1\subseteq\Fix (K_1)$ and $S_2 \subseteq \Fix (K_2)$. Note that the action of $G/K$ on a $K$-fixed surface $S$ has an even number of fixed points, equal to $2 + \dim_{\bk} H^1(G/K; H^1(S))$.
 
 We now restrict the $G$-action to any index $p$ subgroup $K \leq G$, and let $\Fix(K)$ denote its fixed set. The remarks above  show that $\Fix(K)$ contains fixed orientable surfaces, each with an effective action of $G/K \cong \cy p$.  Since a $\cy p$ action on an orientable surface $S \neq S^2$ induces an effective action on $H^1(S)$, we see that the map $H^1(M) \to H^1(\Fix(K))$ must be zero: either all the surfaces are $2$-spheres, so that $H^1(\Fix(K)) = 0$,  or the  $G/K$-action on $H^1(M)$ would be non-trivial, contradicting our homologically trivial assumption.
 
 Therefore, if $H_1(M) \neq 0$ the Borel spectral sequence for $E_K(M)$ does not collapse with $\bZ$-coefficients (by Proposition \ref{prop:fourone}). Since the homology of $M$ is torsion free, $H^r(M) \otimes \Fp \cong H^r(M;\Fp)$, and it follows from the Bockstein sequence that the maps $H^r(K; H^s(M)) \to H^r(K; H^s(M; \Fp)) $ are injective for all  $r>0$. Therefore the Borel spectral sequence for $E_K(M)$ does not collapse with $\Fp$-coefficients either.  Hence if $H_1(M) \neq 0$, the Borel spectral sequence for $E(M_G)$ does not collapse (see Remark \ref{rem:seventwo}).

  If $H_1(M) = 0$, then our assumption that the fixed set $F \neq \emptyset$ and multiplicativity implies that the Borel spectral sequence for $E(M_G)$ does collapse (since no differentials can hit the line $E_2^{*, 0}$). 
\end{proof}

\section{The proof of Theorem B}\label{sec:six}
Let $G = \cy p \times \cy p$, for $p$ odd, and recall that the mod $p$ cohomology algebra
$$H^*(G) = \Fp[u_1, u_2]\otimes \Lambda(x_1, x_2)$$
where $|u_i| = 2$ and $|x_i| = 1$, with $x_i^2 = 0$. We will use cohomology with $\Fp$ coefficients throughout this section.

The \emph{essential cohomology}, denoted $\Ess^*(G) \subset H^*(G)$ is  defined as the intersection of the kernels of the restriction maps induced by the $(p+1)$ non-trivial cyclic subgroups $K \leq G$. A nice description is given by 
\begin{theorem}[{\rm Aksu and Green \cite{Altunbulak-Aksu:2009}}] For $G = \cy p \times \cy p$, the  essential cohomology 
$\Ess(G)$ is the smallest ideal in $H^*(G)$ containing $x_1x_2$ and closed under the action of the Steenrod algebra. Moreover, as a module over $\Fp[u_1, u_2]$, the essential ideal $\Ess^*(G)$ is free on the set of M\`ui generators.
\end{theorem}
This statement is a special  case of their general result. For the rank two case, the M\`ui generators are as follows:
$$\gamma_1 = x_1x_2, \ \gamma_2 = x_1u_2 - x_2u_1, \ \gamma_3 = x_1u_2^p - x_2u_1^p, \text{\ and\ } \gamma_4 = u_1u_2^p - u_2u_1^p.$$
We note that the degrees are $2, 3, 2p+1, 2p+2$ respectively. For detailed calculations, it is useful to let $R :=\Fp[u_1, u_2]$ and $\Lambda := \Lambda(x_1, x_2)$. These are graded rings with $\dim R^{2k} = k+1$, $\dim \Lambda^{1} = 2$ and $\dim \Lambda^{2} = 1$. In this notation, $H^{2k}(G) = R^{2k} \oplus (R^{2k-2} \otimes \Lambda^2)$ and 
$H^{2k+1}(G) =  R^{2k} \otimes \Lambda^1$. Note that $H^*(G)$ is generated as an $R$-module by the  cohomology groups $H^k(G)$, for $k \leq 2$.

\medskip
The proof of Theorem B is based on a detailed study of the Borel spectral sequence. Here is an example for the case $p=3$ which illustrates some of the features.  As explained in that proof, the images of any differentials in the Borel spectral sequence for $H^*_G(M)$ with range $E_r^{k,0}$, for any $k \geq 0$, must belong to $\Ess^*(G)$.

\begin{example}\label{ex:cp2} Let $M = \CP^2$ with the pseudo-free action of $G = \cy 3 \times \cy 3$ given by $S(z_1, z_2, z_3) = (z_1, \omega z_2, \omega^2z_3)$ and $T(z_1, z_2, z_3) = (z_2, z_3, z_1)$. The singular set consists of 12 points, arranged in 4 triangles each fixed by one of the 4 subgroups of order three in $G$. By \cite[Proposition 2.1]{Edmonds:1998}, we have an isomorphism $H^q_G(M) \to H^q_G(\Sigma  )$, for $q > 4$, and hence
$\dim H^q(G; H^0(\Sigma)) = 4$ for $q >4$ (see the proof of Theorem B for details). It turns out that for this dimension bound to hold, the M\' ui generators $\gamma_2$, $\gamma_3$ and $\gamma_4$ (in degrees $3, 7, 8$ respectively) must be hit by differentials. We will use the dimension bound $\dim E_\infty^{q,0} \leq \dim H^q_G(M) = 4$, for $q >4$.

Since $d_2 = 0$, $E_2 = E_3$ and the $E_3$-page has three lines, where the differentials are determined by the values of  $d_3^{0,q}\colon E_2^{0,q} \to E_2^{3, q-2}$, for $q = 2, 4$, and the multiplicative structure of $H^*(G)$.  Let $z \in H^2(\CP^2;\bZ)$ be the generator dual to the homology class of $\CP^1 \subset\CP^2$, and let $w = z^2 \in H^4(\CP^2;\bZ)$ be the orientation class. Then $d_3(z) = \gamma_2$ and 
$d_3(w) = -\gamma_2 z$. Therefore $d_3^{2,4}(\gamma_1w) = -\gamma_1\gamma_2 = 0$ and  $d_3^{3,4}(\gamma_2 w) = - (\gamma_2)^2 z = 0$ so these elements persist to the $E_5$-page, with $\dim E_5^{2,4} = \dim E_5^{3,4} = 1$.

For the dimension count of $H^6_G(M)$ we also need to compute $E_4^{4,2}$ and $E_4^{6,0}$ (and note that $E_4 = E_5$).
It is not hard to check that $\image d_3^{1,4} = \la \gamma_1u_1, \gamma_1u_2\ra \subset E_3^{4,2} \cong H^4(G)$, and this equals the kernel of $d_3^{4,2}$. Therefore $E_4^{4,2}=0$. Next, 
$$\image d_3^{3,2} = \Lambda^1 \cdot R^2 \cdot \gamma_2 =  \la \gamma_1u^2_1, \gamma_1u_1u_2, \gamma_1u^2_2 \ra \subset E_3^{6,0} \cong H^6(G).$$
Therefore $\dim E_4^{6,0} = 4$ and the dimension count shows that there is one remaining non-zero differential $d_5^{2,4}\colon E_5^{2,4} \to E_5^{7,0}$ affecting the line $k+l = 6$.

For the dimension count of $H^7_G(M)$ we have $\dim E_4^{3,4} = 1$ and we make similar calculations to determine $E_4^{5,2}$ and $E_4^{7,0}$. We see that $\image d_3^{2,4} = \la \gamma_2 u_1, \gamma_2u_2\ra = \ker d_3^{5,2}$, so $E_4^{5,2} = 0$.  We compute
$$\image d_3^{4,2}  =  \gamma_2 \cdot R^4 =  \la x_1u_1^2u_2 - x_2u_1^3, x_1u_1u_2^2 - x_2u_1^2u_2, x_1u_2^3 - x_2u_1u_2^2 \ra \subset H^7(G).$$
Therefore $\dim E_4^{7,0} = 5$, and the dimension count confirms that  $d_5^{2,4}$ is non-zero with 1-dimensional image.
More precisely,  $\image d_5^{2,4} =\la \gamma_3\ra$ since 
$$ 0 \neq \image d_5^{2,4} \subseteq \Ess^7(G)/\image d_3^{4,2} = \la \gamma_3, \gamma_2 \cdot R^4\ra /\image d_3^{4,2}  \cong \la \gamma_3 \ra.$$
By a similar calculation,  $\image d_3^{5,2} = \image d_3^{4,2} \cdot \{x_1, x_2\} = \Lambda^2 \cdot R^6$ has dimension 4. and  $\dim E_4^{8,0} = 5$, so that $d_5^{3,4}$ must have 1-dimensional image.
Since $\Ess^8(G) = \la  \gamma_4, \Lambda^2 \cdot R^6, \gamma_4\ra$, it follows that $d_5^{3,4}$ hits $\gamma_4$, and hence the differentials surject onto $\Ess^q(G)$, for $q>2$.
\end{example}
   
 \begin{remark}
   To rule out higher rank actions as asserted in Theorem B, we will show that the
   M\`ui generators $\gamma_{2p+1}$ and $\gamma_{2p+2}$ for $p \geq 5$, can not  be hit by differentials in the Borel spectral sequence for $G= \cy p\times \cy p$. This would imply that the groups $H^q_G(M)$ for large values of $q$ would have dimensions contradicting the bound \eqref{eq:dimcount} from the singular set, and hence rule out the existence of these actions.
    \end{remark}

   In order to prove this claim, the key point is that the differentials are determined through multiplicativity by their values on $E_r^{k, l}$ for $k\leq 3$. This is a consequence of the structure of the cohomology ring $H^*(G)$, which is generated by clasees in degrees $\leq 2$ (as explained in Example \ref{ex:cp2}).
 
\begin{proof}[The proof of Theorem B]
Suppose that $G$ is acting homologically trivially on $M$ with $\chi(M) \neq 0$. In addition, we are assuming that the action is pseudofree, meaning that the singular set $\Sigma$ is a discrete set of points. Note that $M^G = \emptyset$ since $G$ can not act freely on $S^3$. Each subgroup $K \cong \cy p$ has $\chi(M) >0$ fixed points, which are then permuted in $\chi(M)/p$ orbits of size $p$ by $G/K$, so that $H^0(\Fix(K))$ is the direct sum of $\chi(M)/p$ copies of the permutation $G$-module $\Fp[G/K]$.

By \cite[Proposition 2.1]{Edmonds:1998}, we have an isomorphism 
$H^q_G(M) \xrightarrow{\approx} H^q_G(\Sigma  )$, for $q > 4$, and this provides a dimension count as above. 
 In this case, we have $p \mid \chi(M)$ and there are $p+1$ distinct subgroups of order $p$ in $G$, so that
 \eqncount
 \begin{equation}\label{eq:dimcount}
\dim H^q(G; H^0(\Sigma)) =\sum_i \dim H^q(G; \Fp[G/K_i])^{\chi(M)/p} =  \frac{\chi(M)}{p} \cdot (p +1)
\end{equation}
by Shapiro's Lemma. The main observation is that the images of any differentials in the Borel spectral sequence for $H^*_G(M)$ with range $E_r^{k,0}$, for any $k \geq 0$, must belong to $\Ess^*(G)$. This follows immediately by comparing the spectral sequences for $H^*_G(M)$ and $H^*_G(\Sigma)$. Similarly, by Proposition \ref{prop:cyclic} the images of the higher differentials $d_r$, for $r \geq 3$, must lie in $\Ess^*(G)$ modulo indeterminacy from the earlier differentials.  Moreover, since the  $\Res_K\colon  H^r(G; \Fp[G/K]) \to H^r(K; \Fp[G/K])$ is an injection,
 the sum of the restriction maps 
$$ \bigoplus\nolimits_K \Res_K\colon  H^q_G(M) \to \bigoplus\nolimits_K\{ H^q_K(M) \vv 1\neq K \neq G\}$$
is also an injection for $q > 4$. 

We have  commutative diagram (for $q > 4$):
$$\xymatrix{H^q(G) \ar[r]^{\cong}\ar[d]^{\Res_K}& E_2^{q,0}(M_G)\ar[d]^{\Res_K} \ar@{->>}[r]
& {E_\infty^{q,0}(M_G)}\ar @{ >->} [d]^{\Res_K}\ar@{ >->}[r]
& {H^q_G(M)}\ar@{ >->}[d]^{\Res_K}\\
\bigoplus H^q(K) \ar[r]^{\cong}& \bigoplus  E_2^{q,0}(M_K) \ar[r]^\cong& \bigoplus E_\infty^{q,0}(M_K)\ar[r]^\cong&  \bigoplus H^q_K(M)
}$$

It follows from this diagram, and the fact that the images of differentials with range in $E_r^{k,0}$ are contained in $\Ess^*(G)$,  that $\Ess^q(G) = \ker\{ E_2^{q,0}(M_G) \twoheadrightarrow  E_\infty^{q,0}(M_G)\}$, for $q >4$ is a necessary condition for the $G$-action to exist.

For $p = 3$, the M\`ui generators have dimensions 2, 3, 7, 8 and these are all within the range of the differentials $d_r^{k,l}$, for $r \leq 5$ and $k \leq 3$ (as in Example \ref{ex:cp2}). However, for $p >5$, only the first two M\`ui generators $\gamma_1 = x_1x_2$ and $\gamma_2 = x_1u_2 - x_2u_1$ can be hit by a non-zero differential $d^{k,r-1}_r$ if $k = 2p+1 -r$ or $k = 2p+2-r$, with $r\leq 5$. Since  $2p+1\geq 11$, this implies $k \geq 6$. 

\bigskip
Consider the  differentials $d_r$ with range in the line $E_r^{*,0}$. These are $d_2^{k,1}$, $d_3^{k,2}$, $d_4^{k,3}$ and $d_5^{k,4}$. At each page, if  the differential $d_r^{k,r-1}$ is non-zero its image must lie in $\Ess^r(G)$. 
We claim that the images of the differentials $d_r^{k, r-1}$, for all $k \geq 0$,  will be contained in the module generated by the first two  M\`ui generators $\gamma_1$ and $\gamma_2$ under the action of the polynomial algebra $\Fp[u_1, u_2]$. Since $\Ess^*(G)$ is a free module on all the M\`ui generators (by \cite[Theorem 1.2]{Altunbulak-Aksu:2009}), we will have a contradiction to the dimension bound on $H^q_G(M)$ for large $q$,  and the assumed $G$ action does not exist. 

\medskip
To verify this, we tabulate the generators of $\Ess^k(G)$ for $2\leq k \leq 6$ as follows: 
$$\Ess^k(G) = \{ \la \gamma_1\ra , \la \gamma_2\ra, \la \gamma_1u_1, \gamma_1u_2\ra,
\la \gamma_2 u_1, \gamma_2u_2\ra, \la \gamma_1u_1^2 , \gamma_1u_1u_2 , \gamma_1u_2^2\ra \}.$$
For use in our arguments below, we also note that $\Ess^k(G)$ is generated by $\gamma_1$ and $\gamma_2$ over $R$ in degrees $k \leq 10$ (for all primes $p \geq 5$).

\medskip
We first fix some notation for an $\Fp$-basis of the cohomology of $M$:  let us denote them by  $w \in H^4(M)$,  
$\la \beta_1, \dots, \beta_t\ra \subset H^3(M)$, $\la z_1, \dots, z_s \ra \subset H^2(M)$, and $\la \alpha_1, \dots, \alpha_t\ra \subset H^1(M)$. We will check the images of the differentials $d^{k,r-1}_r$ in each case.

\medskip
\noindent
\textbf{The image of $d_2^{k,1}\colon E_2^{k,1} \to E_2^{k+2,0}$}. Since $\image d_2^{0,1} \subseteq \Ess^2(G) = \la \gamma_1 \ra$, either $d_2^{k,1} = 0$, for $k \geq 0$, or  $d_2^{0,1}(\alpha_1) = \gamma_1$ and we may assume that $d_2^{k,1}(\alpha_k)=0$, for $k \geq 2$. In the second case,
 $\image d_2^{k,1} \subseteq \gamma_1 \cdot R$ and $\ker d_2^{k,1} = \la \alpha_1 \cdot (\Lambda^1\otimes R),  \alpha_2, \dots, \alpha_t\ra$. In particular, the image of $d_2^{k,1}\colon E_2^{k,1} \to E_2^{k+2,0}$ does not contain $\gamma_2$, $\gamma_3$ or $\gamma_4$.

\medskip
\noindent
\textbf{The image of $d_3^{k,2}\colon E_3^{k,2} \to E_2^{k+3,0}$}. The image of $d_2^{k,2}$ restricted to any order $p$ subgroup of $G$ must be surjective, by Proposition \ref{prop:cyclic}. It follows that either $d_2^{0,2}(z_i) \neq 0$ and projects non-trivially to $\alpha_j \cdot H^2(G)/\la \gamma_1\ra $, for some $\alpha_j$, or 
$d_2^{0,2}(z_i) = 0$ and $\image d_3^{k,2}(z_i)  \subseteq  \gamma_2 \cdot R$. Therefore $ \image d_3^{k,2}$ does not contain  $\gamma_3$ or $\gamma_4$.

\medskip
\noindent
\textbf{The image of $d_4^{k,3}\colon E_4^{k,3} \to E_2^{k+4,0}$}. Since $d_2^{k,3}$ is injective when restricted to any order $p$ subgroup, by Proposition \ref{prop:cyclic}, it follows that $\image d_2^{k,3}(\beta_i)$ projects non-trivially to $H^k(G; H^2(M))/\la \Ess^k(G)\cdot H^2(M)\ra$. Therefore $d_2^{k,3}$ is injective, and $E_r^{k,3} = 0$ for $r \geq 3$ implies $d_4^{k,3}=0$.

\medskip
\noindent
\textbf{The image of $d_5^{k,4}\colon E_4^{k,4} \to E_4^{k+5,0}$}. Since $d_2^{k,3}$ is injective, we have
$d_2^{k,4} = 0$, for $k \geq 0$. Suppose first that $0 \neq d_3^{0,4}(w) \in \gamma_2 \cdot H^2(M)$. Then $\image d_3^{k,4} \subseteq (\la \gamma_1, \gamma_2\ra \cdot R)\cdot H^2(M)$. 
Therefore $\ker d_3^{k,4} \subseteq \la \overline{\gamma_1w}, \overline{\gamma_2w}\ra \cdot R\subseteq E_4^{k,4}$, and $\image d_4^{k,4}$ is generated by the images $d_4^{2,4}(\overline{\gamma_1w})\in \Ess^6(G) \cdot E^{6,1}_4$ and $d_4^{3,4}(\overline{\gamma_2w})\in \Ess^7(G) \cdot E^{7,1}_4$ under the action of $R$. 

 If both these images under $ d_4^{k,4}$ are non-zero, then $\ker d_4^{k,4}= 0$  since multiplication by elements of $R$ is injective on $\image d_4^{k,4}$. Therefore $d_5^{k,4} = 0$, hence
 $\gamma_3$ or $\gamma_4$ can not be hit.

If either of these images under $ d_4^{k,4}$ are zero, then their corresponding images under $d_5^{k,4}$ will be contained in $\Ess^q(G)$ for $q \leq 7$, and  again $d_5$ can not hit $\gamma_3$ or $\gamma_4$.

\bigskip
For $p=3$, we rule out actions of $G = \cy 3 \times \cy 3 \times \cy 3$ by similar arguments. In the rank three case, there are eight
  M\`ui generators, starting with $\gamma_1 = x_1x_2x_3$ and $\gamma_2 = \beta(\gamma_1)$ in degrees 3, 4, and continuing in degrees 8, 9, 20, 21, 25, 26 (see  \cite[Section 3]{Altunbulak-Aksu:2009}). The higher M\`ui generators are outside the range of differentials hitting the line $E_r^{*, 0}$. Hence such an action does not exist. 
  
 For $p=2$ and $G = \cy 2 \times \cy 2 \times \cy 2$, the cohomology ring  is now $H^*(G) = \bbF_2[x_1, x_2, x_3]$ and there is just one  M\`ui generator 
 $$\gamma = x_1x_2x_3(x_1 + x_2) (x_1 + x_3) (x_2 + x_3) (x_1 + x_2 + x_3)$$
 in degree 7, which is the product of the distinct linear forms. The ideal $\Ess^*(G) = \la \gamma\ra$ is a free module over $\Fp[x_1, x_2, x_3]$ and $\Ess^*(G)$ is the Steenrod  closure of $\gamma$ in $H^*(G)$ (see  \cite[Lemma 2.2]{Altunbulak-Aksu:2009}). This means that the rank two actions can not be ruled out by the method above (in fact such actions exist on $S^2 \times S^2$). 
 
 However, we can use the information contained in the proof of Proposition \ref{prop:cyclic} to see that the images of the differential $d_2^{0, 2}$ in $E_2^{2,1}(K)$ must be non-zero in each summand of $H^2(K) \otimes H^1(M)$, and for each subgroup $K\cong \cy 2$. Therefore, there must be a class $\alpha \in H^2(G)$ such that $\Res_K(\alpha) \neq 0$ for each $K < G$ of order two. We claim that no such class exists. To see this, let $H  \cong \cy 2 \times \cy 2$ be an index two subgroup. The only possibility for $\Res_H(\alpha)$ is the class $\delta = \bar x_1^2 + \bar x_1 \bar x_2 + \bar x_2^2$, where $\bar x_i$ denote the degree 1 generators of the cohomology of $H$. We look at the restriction of a general element
 $$ \alpha = \sum_{1\leq i \leq 3} a_i x_i ^2+ \sum_{i <j} b_{ij} x_i x_j \in H^2(G)$$
 to each of the index two subgroups $H$ obtained by imposing one of the 7 linear relations in the formula for $\gamma$.  First, to get $\Res_H(\alpha) = \delta$ by setting $x_i = 0$ for each $1\leq i \leq 3$ separately,  we find that all the coefficients $a_i$ and $b_{ij}$ must be non-zero. But then, setting $x_1 + x_2 = 0$ gives $\Res_H(\alpha) = \bar x_1^2 + \bar x_3^2 \neq \delta$. Hence $\alpha$ does not exist, and such a rank three  pseudofree $G$-action is ruled out.
  \end{proof}

\section{Some Examples}
\label{sec:seven}

In this section, we give some illustrative examples of group actions on a closed, connected oriented $4$-manifolds. These indicate the necessity of the conditions in  Theorem \ref{thm:threeone} for the Borel spectral sequence to collapse. 
We let $\bk = \Fp$ with the prime  $p$ under consideration understood.

\begin{example}\label{ex:one}
Consider (i) $S^1\times S^3$ with $\cy 3$ acting trivially on $S^1$ and by rotation on $S^3$, so that the fixed point set $S^1\times S^1$,  and (ii)  $\bbC P^2$ with  a $\cy 3$-action fixing $\bbC P^1$ and a point. 
Taking the equivariant connected sum along the two dimensional fixed set, we get $M=S^1\times S^3\#\bbC P^2$ with the fixed point set $F=S^1\times S^1\, \#\, \bbC P^1\cup \{pt\}$.

By  Theorem \ref{thm:threeone} since $H_1(F)=\bbZ\oplus\bbZ$ surjects onto $ H_1(M)=\bbZ$ , the Borel spectral sequence with integral coefficients collapses for this example. 
Since the action is homologically trivial, and
 the total dimensions satisfy $$\sum_r  \dim_{ \bk}H^r(F)=5=\sum_r  \dim_{ \bk} H^r(M) $$ the Borel spectral sequence with $\bbF_3$ coefficients collapses by Proposition \ref{prop:twotwo}.
\end{example}
Next, we have a case where the fixed point set consists of isolated points and $H_1(M)$ is torsion free.
\begin{example}\label{ex:two}
Consider the diagonal action of $\cy p$ on $S^2\times S^2$ with four fixed points. Now take two copies of  $S^2\times S^2$ with this action and take the equivariant connected sum along two pairs of fixed points where the representations of the tangent bundles are equivalent. We obtain a $4$-manifold $M$ which has a $\cy p$-action with four fixed points. $M$ has $H_i(M)=\bZ$ for $i=0,1,3,4$ and $H_2(M)=(\bbZ)^4$ as  homology groups.  Since the action is homologically trivial, we can also again use Proposition \ref{prop:twotwo}:
$$ \sum_r  \dim_{ \bk} H^r(M^G)=4\neq 8=\sum_r  \dim_{ \bk} H^r(M) $$ showing that the Borel spectral sequence with $\Fp$ coefficients does not collapse. 
\end{example}
There are also examples where the fixed point set is two dimensional, but the Borel spectral sequence does not collapse:
\begin{example}\label{ex:three}
Consider again a $ \cy  3$ action on $\bbC P^2$ fixing a $\bbC P^1$ and a point. Take two copies of this and take the equivariant connected sum along the two dimensional fixed sets and the fixed points. The manifold we obtain is a $4$-manifold having $ \cy  3$ action with a connected two dimensional fixed set which has the homology of the two sphere.  Again the action is homologically trivial and by Proposition \ref{prop:twotwo}:
$$\sum_r  \dim_{ \bk} H^r(M^G)=2\neq 6=\sum_r  \dim_{ \bk} H^r(M) $$ showing that the Borel spectral sequence with $\bbF_3$ coefficients does not collapse. Here the map $H_1(F)\to H_1(M)$  is not surjective. \end{example}

Here is an example with $p$-torsion in $H_1(M)$.

\begin{example}\label{ex:four}
Let $M = L^3(\cy p, 1) \times S^1$, with the action of $G = \cy p$ given by 
$$\zeta\cdot([z_1: z_2], z_3) = 
([\zeta\cdot z_1: z_2], z_3).$$
 Note that $[\zeta\cdot z_1: z_2] = [ z_1: \zeta^{-1}\cdot z_2]$ because of the equivalence relation used to define $L^3(\cy p, 1)$. The fixed set $F = S^1 \times S^1 \disjointunion S^1\times S^1$, and $H_1(F) \to H_1(M)$ is surjective, hence the Borel spectral sequence collapses.
\end{example}

Here is an example for which $H_1(M)$ has non-trivial $G$-action.
\begin{example}
For $G = \cy 2$ , consider the diagonal reflection on $M=S^1\times S^3$ which reverses the orientation on each factor. The fixed point set  $ F=S^2\bigsqcup S^2$ and $H^1(M)=\bbZ_{-}$. Because the total dimensions satisfy $$ \sum_r  \dim_{ \bk}H^r(M^G)=4=\sum_r  \dim_{ \bk} H^r(M), $$ the Borel spectral sequence collapses.
\end{example}

%

Finally we will give an example with $G=\cy p\times\cy p$ acting homologically trivially.

\begin{example}
Consider the $\cy p\times\cy p$-action on $S^2\times S^2$ given by the product of two rotation actions of $\cy p$ on $S^2$. This action has 4 fixed points and singular set consisting of four $2$-spheres. Let $M$ be obtained by taking the equivariant connected sum of two copies of  $S^2\times S^2$ along two of the fixed points. Then $M$ admits  a $\cy p\times\cy p$-action with 4 global fixed points and which is homologically trivial and locally linear (in fact smooth).
 The Borel spectral sequence with $\Fp$ coefficients does not collapse (by Remark \ref{rem:seventwo} and Corollary \ref{cor:sevenone}). This is a counter-example  to \cite[Corollary 3.2]{McCooey:2013}.
\end{example}

\providecommand{\bysame}{\leavevmode\hbox to3em{\hrulefill}\thinspace}
\providecommand{\MR}{\relax\ifhmode\unskip\space\fi MR }
\providecommand{\MRhref}[2]{%
  \href{http://www.ams.org/mathscinet-getitem?mr=#1}{#2}
}
\providecommand{\href}[2]{#2}

\end{document}